\documentclass[11pt,a4,twoside]{amsart}
\usepackage[latin1]  {inputenc}%
\usepackage[T1]      {fontenc }%
\usepackage          {latexsym}%
\usepackage          {amsmath }%
\usepackage          {amsfonts}%
\usepackage          {amssymb }%
\usepackage          {amsthm  }%
\usepackage[dvips]   {graphicx}%
\usepackage          {wrapfig }%
\usepackage          {mathdots}%
\usepackage          {fancyhdr}%
\usepackage          {a4      }%
\usepackage          {psfrag  }%
\usepackage          {url     }%
\usepackage          {paralist}%
\usepackage          {color   }%
\usepackage          {calc    }%
\usepackage          {units   }%
\usepackage[small,it]{caption }%
\usepackage{booktabs}
\usepackage{tabularx}
\usepackage{subfigure}
\usepackage{bbm}
\DeclareMathOperator{\conv }{conv }%

\newcommand{\R}{\mathbb R}%
\newcommand{\Z}{\mathbb Z}%

\allowdisplaybreaks[4]%

\theoremstyle{plainItalics}
\newtheorem{theorem}{Theorem}[section]

\newtheorem{corollary}[theorem]{Corollary}
\newtheorem{prop}[theorem]{Proposition}
\newtheorem{proposition}[theorem]{Proposition}

\newtheorem{definition}[theorem]{Definition}

\theoremstyle{plainNoItalics}

\numberwithin{theorem}{section}%
\numberwithin{figure}{section}%
\numberwithin{table}{section}%

\DeclareMathOperator{\aff}{aff}
\DeclareMathOperator{\DP}{DP}

\title{Polytopes associated to Dihedral Groups}

\author[Baumeister]{Barbara Baumeister}
\address{Barbara Baumeister, Universit\"at Bielefeld, Germany}
\email{b.baumeister@math.uni-bielefeld.de}
\author[Haase]{Christian Haase}
\address{Christian Haase, Goethe-Universit\"at Frankfurt, Germany}
\email{haase@mathematik.uni-frankfurt.de}
\author[Nill]{Benjamin Nill}
\address{Benjamin Nill, Case Western Reserve University, Cleveland, OH, USA}
\email{benjamin.nill@case.edu}
\author[Paffenholz]{Andreas Paffenholz}
\address{Andreas Paffenholz, Technische Universit\"at Darmstadt, Germany}
\email{paffenholz@mathematik.tu-darmstadt.de}

\date{\today}

\begin{document}

\maketitle

%\abstract{
\begin{abstract}
In this note we investigate the convex hull of those $n \times
n$-permutation matrices that correspond to symmetries of a regular
$n$-gon. We give the complete facet description.
As an application, we show that this yields a Gorenstein polytope, and
we determine the Ehrhart $h^*$-vector.
\end{abstract}
%}
\section{Introduction}

To any finite group $G$ of real $(n\times n)$-permutation matrices we
can associate the {\em permutation polytope} $P(G)$ given by the
convex hull of these matrices in the vector space $\R^{n \times n}$.
A well-known example of such a polytope is the Birkhoff polytope
$B_n$, which is 
defined as the convex hull of all $(n\times n)$-permutation
matrices~\cite{0355.15013,0839.52007}. This polytope appears in various
contexts in mathematics from optimization to statistics to enumerative
combinatorics. (See, e.g.,
\cite{0581.05038,Onn93,Pak00,1082.90081,1077.52011}.)
It is also a famous example of a Gorenstein polytope (see Section 5).
Gorenstein polytopes turn up in connection to Mirror Symmetry in
theoretical physics.

Guralnick and Perkinson \cite{1108.52014} studied polytopes associated
to general subgroups $G$ and proved results about their dimension, and
about the diameter of their vertex-edge graph. A systematic exposition
of general permutation polytopes is given in~\cite{Baumeister2007}.
There, we studied which groups lead to affinely equivalent polytopes,
we considered products of groups and polytopes, classified
low-dimensional cases, and we formulated several open conjectures.

In order to get an intuition about what one can expect in general, it
is instructive to consider some special classes of permutation groups.
A seemingly very difficult case is when $G$ equals the group of
even permutation matrices.
Just to exhibit exponentially many
facets is already a daunting task, for this see~\cite{1103.52010}.
Even for cyclic $G$ we showed in~\cite{1109.0191} that these polytopes
have a surprisingly complex and not yet fully understood facet
structure.

In \cite{CP04} Collins and Perkinson studied polytopes given by
Frobenius groups.  A special case is the dihedral group $D_n$ for $n$
odd, which was considered in more detail by Steinkamp
\cite{Ste99}. Since $D_n$ is the automorphism group of a regular
$n$-gon, the cases where $n$ is even and odd are quite different. 

The most recent paper on permutation polytopes \cite{loera} focused on
determining the volumes of permutation polytopes associated to cyclic
groups, dihedral groups, and Frobenius groups. In order to compute the
volume of $P(D_n)$, the authors find a Gale dual combinatorial
description, which they use to provide an explicit formula for the
Ehrhart polynomial of $P(D_n)$.

The dihedral group $D_n$ is the automorphism group $\mathrm{Aut}(C_n)$
of a cycle $C_n$, and any permutation matrix $M(\sigma)$ of an element
$\sigma\in D_n$ commutes with the adjacency matrix $A$ of $C_n$. So
any point in $P(D_n)$ commutes with $A$, and
\begin{align*}
  P(D_n)\subseteq \{M\in\R^{n\times n}\mid M\text{ is doubly stochastic
    and } MA=AM\}\,.
\end{align*}
Here, a matrix is doubly stochastic if all entries are non-negative
and each row and column sum is $1$. Tinhofer~\cite{0581.05038} asks,
more generally, for a classification of those undirected graphs $G$
where the two sets above are equal, i.e. where the commutation
condition $MA=AM$ already suffices to characterize the elements of
$P(\mathrm{Aut}(G))$ among all doubly stochastic matrices. The
Birkhoff-von Neumann theorem is the special case where $A$ is the unit
matrix.  Tinhofer shows that this also holds for the adjacency
matrices of cycles and trees~\cite[Thms.\ 2\&3]{0581.05038}.

In this note, we independently investigate $P(D_n)$ in a more direct
and elementary way.  We give a complete list of its facet inequalities
(Theorem~\ref{theo1}, Theorem~\ref{theo2}). As an application, we
observe that these lattice polytopes are Gorenstein polytopes, and we
get a nice description of the generating function of their Ehrhart
polynomials (Theorem~\ref{theo3}, Corollary~\ref{coro}).

\textbf{Acknowledgments:} Many results are based upon experiments and
computations using package \texttt{polymake}~\cite{GJ05} by Gawrilow
and Joswig. The first author likes to thank for support by the DFG
through the SFB 701 ``Spectral Structures and Topological Methods in
Mathematics''. The second author is supported by DFG Heisenberg (HA
4383/4-1). The third author is supported by the US National Science
Foundation (DMS 1102424). The last author is supported by a DFG
Priority Program (DFG SPP 1489).

\section{Notation and preliminary results}

Let $S_n$ be  the permutation group on $n \ge 3$  elements.  Every permutation
$\sigma\in  S_n$  can  be   represented  by  an  $n\times  n$ matrix
$M_\sigma$ with entries $\delta_{i,(j)\sigma}$.  So the entries are 
in $\{0,1\}$  and there is exactly one $1$ in each
row  and  column. Notice that we apply matrices and permutations from the
right.
   We  can  view  such  a  matrix   as  a  vector  in
$\R^{n^2}$. For a subgroup $G$ of $S_n$ we define the polytope
\begin{align*}
  P_G:=\conv(M_\sigma\mid \sigma\in G)\,.
\end{align*}
This is a $0/1$-polytope, so all  matrices are in fact vertices of the
polytope. 

We denote by   $D_n$  the   subgroup  of  $S_n$
corresponding to the symmetry group of the regular $n$-gon, the {\em dihedral  group of order $2n$}. This group
is  generated by  two  elements.  If  $n$  is odd,  then these may taken to be the
rotation $\rho$ of the $n$-gon by $360/n$ degrees, and the reflection $\tau$
along a line through one vertex and the midpoint of the opposite edge.
If $n$ is  even, then the second generator $\tau$ is instead the 
reflection along  a line through two opposite  vertices.  Thus $\rho$ is the permutation $(1,2, \ldots , n)$
and $\tau$ the reflection $(2,n)(3, n-1) \cdots ((n+1)/2, (n+3)/2)$ if $n$ is odd and $(2,n)(3, n-1) \cdots (n/2, (n/2)+2)$
if $n$ is even.

The associated permutation polytope is the convex hull of the corresponding matrices,
\begin{align*}
  \DP_n:=\conv(M_\sigma\mid \sigma \in D_n)\,.
\end{align*}

The dihedral group $D_n$ has $2n$ elements 
$$\rho^0, \rho^1, \rho^2,
\ldots,  \rho^{n-1},  \tau, \tau\rho, \tau\rho^{2},  \tau\rho^{3},
\ldots, \tau\rho^{n-1}.$$
We label the vertices of $ \DP_n$ by  $v_0, \ldots, v_{n-1},
w_0, \ldots, w_{n-1}$ in this order. Let us give a more convenient way to write these matrices.

Let $I$ be the $n$-dimensional identity matrix and $R$ be the $n\times
n$ matrix that has $0$'s everywhere except at the $n$ entries $(i,j)$,
where $0\le i,j\le n-1$ and $j\equiv i+1\mod n$:
\begin{align*}
  R=\left[
    \begin{array}{ccccccc}
      0&1&0&0&\cdots&0&0\\
      0&0&1&0&\cdots&0&0\\
      \vdots&\ddots&\ddots&\ddots&&\vdots&\vdots\\
      \vdots&&\ddots&\ddots&\ddots&\vdots&\vdots\\
      0&\cdots&\cdots&0&\cdots&1&0\\
      0&0&0&0&\cdots&0&1\\
      1&0&0&0&\cdots&0&0
    \end{array}
    \right].
\end{align*}
Reading the matrices $M_\sigma$ row by row, we can identify $M_\sigma$ with a (row) vector in $\R^{n^2}$. For instance, 
the $2\times 2$ identity matrix would be identified with $(1\ 0 \ 0\  1)$. 
Under this identification the vertices of  $\DP_n$ are
(in the order given above) the rows of the $2n \times n^2$ matrix
\begin{align*}
  \left[
    \begin{array}{ccccc}
      R^0& R^1& R^2& \cdots& R^{n-1}\\
      R^0& R^{-1}& R^{-2}& \cdots& R^{-(n-1)}
    \end{array}
  \right]\,.
\end{align*}
Permuting the coordinates (corresponding to a linear automorphism of $\R^{n^2}$) we may write the vertices in the form 
\begin{equation}
 V= \left[
    \begin{array}{ccccc}
      I& I& I& \cdots& I\\
      I& R^{-2}& R^{-4}& \cdots& R^{-2(n-1)}
    \end{array}
  \right]\,.
\label{main-eq}
\end{equation}
Clearly, the first $2n$ coordinates of the vertices linearly determine
the remaining coordinates. So we can project onto $\R^{2n}$ without
changing the combinatorics of the polytope. Hence, we observe that the
dimension of $\DP_n$ is at most $2n$.

\section{The situation for odd $n$}

In this section we complety describe $\DP_n$ for $n$ odd. As it will turn out, 
it is useful to introduce a new polytope that will serve as a basic building block for
both situations of even $n$ and odd $n$. 

\begin{definition}{\rm Let $Q_n$ be the polytope defined as the convex
hull of the rows of the $2n \times n^2$-matrix
\begin{align}
  W:=\left[
    \begin{array}{lllll}
      I&I&I&\ldots&I\\
      I &R^1&R^2&\ldots&R^{n-1}
    \end{array}\right]\,.\label{eq:qnvert}
\end{align}}
\end{definition}

While $Q_n$ differs from $\DP_n$ for even $n$, 
for odd $n$ the $R^{2k}$ for $0\le k \le n-1$ are a permutation of the
$R^k$ for $0\le k \le n-1$. So we deduce from
(\ref{main-eq}) that, for $n$ odd, $Q_n$ is up to a permutation of
coordinates just the polytope $\DP_n$.

\begin{proposition}
  For odd $n$, the polytopes $\DP_n$ and $Q_n$ are affinely
  isomorphic.\qed
\end{proposition}

The following theorem examines the structure of $Q_n$ for arbitrary
$n$. For $n$ odd, this result is a special case of Thm.~4.4 in
\cite{CP04}.

Let us fix some convenient notation. We denote by $\Delta_r$ the
$r$-dimensional simplex.  We also use for any two integers $s,k$, the
term $[s]_k \in \{0, \ldots, k-1\}$ to denote the remainder of $s$
upon division by $k$. The {\em free sum} of two
polytopes $P$ and $P^\prime$ of dimensions $d$ and $d^\prime$ is the polytope
$$P \oplus P^\prime:= \conv(\{(p,0) \in \R^{d+d^\prime}~|~p \in P\} \cup \{(0,p^\prime) \in \R^{d+d^\prime}~|~p^\prime \in P^\prime\}).$$

\begin{theorem}[Collins\&Perkinson~\cite{CP04}] Let $n$ be even or odd.
  The polytope $Q_n$ has dimension $2n-2$ and is a free sums of two
  copies of $\Delta_{n-1}$. Taking coordinates $x_0, \ldots, x_{n^2-1}$ for $\R^{n \times n}$, 
its affine hull is given by the
  equations
  \begin{align}
    1&=\sum_{i=ln}^{(l+1)n-1}x_i\tag{aff}\label{eq:aff}\\
    0&=x_{kn+[j]_n}-x_{(k+1)n+[j]_n}
    -x_{(k+1) n+[j+1]_n}+x_{(k+2) n+[j+1]_n}\tag{$A_{j,k}$}\label{eq:Ajk}
  \end{align}
  for $0\le l\le n-1$, $0\le j\le n-2$, $0\le k\le n-3$.

  An irredundant  system of inequalities defining  the polytope inside
  its affine hull is given by the inequalities
  \begin{align*}
    x_i\ge 0
  \end{align*}
  for $0\le i\le n^2-1$.
\label{theo1}
\end{theorem}
\begin{proof}
  All the given equations are satisfied by the vertices of $Q_n$.
  There are $n$ equations of type (\ref{eq:aff}) and $n^2-3n+2$
  equations of type (\ref{eq:Ajk}).  They are easily seen to be linearly
  independent, so the dimension of $Q_n$ is at most $2n-2$.  On the
  other hand,
%  the affine hull of $Q_n$ does not contain the origin, % wozu?? C
  deleting any row of $W$ leaves us with a linearly (and hence
  affinely) independent set of row vectors.
  (Observe that deleting a row leaves us with a column that contains
  exactly one $1$.) Hence, $\dim(Q_n)=2n-2$ and the given equations
  define the affine hull of $Q_n$ in $\R^{n^2}$.

  Further, we see that any $2n-1$ of the $2n$ rows of $W$ span the
  affine hull of $Q_n$. So  any facet of $Q_n$  has $2n-2$
  vertices. Since the inequalities $x_j\ge  0$ are $0$ on exactly
  $2n-2$ of the rows, they all define facets. 

  In order to prove that $Q_n$ is a free sum of simplices we observe
  that the first $n$ and the last $n$ vertices define
  $(n-1)$-dimensional simplices sitting in transversal subspaces
  (intersecting in the matrix corresponding to the row vector $(1/n,
  \ldots, 1/n)$). Therefore, the combinatorial dual of $Q_n$
  corresponds to the product of $\Delta_{n-1}$ with itself. In
  particular, $Q_n$ has precisely $n^2$ facets, so the facet
  description given above is complete.
\end{proof}

\section{The situation for even $n$}

Recall that the \emph{join} $P\star Q$ of two polytopes $P$ and $Q$ is
the convex hull of $P \cup Q$ after embedding $P$ and $Q$ in skew
affine subspaces. The dimension of $P\star Q$ equals
$\dim(P)+\dim(Q)+1$. For instance, the join of two intervals is a
tetrahedron.

\begin{theorem}
  Let $n$ be even. The polytope $\DP_n$ is a join of two copies of
  $Q_{n/2}$. In particular, its dimension is $2n-3$.
\label{theo2}
\end{theorem}
Combined with Theorem~\ref{theo1}, this result gives a complete
description of the facet inequalities and the affine hull equations of
$\DP_n$ for $n$ even.
\begin{proof}
Permuting the coordinates, we can transform $V$ (see (\ref{main-eq})) into 
\begin{align*}
  \left[
    \begin{array}{ccccccccccc}
      I&I&I&\cdots&I&\quad&I&I&I&\cdots &I\\
      R^0&R^2&R^4&\cdots&R^{n-2}&&R^0&R^2&R^4&\cdots&R^{n-2}
    \end{array}
  \right]
\end{align*}
Clearly, projecting  onto  the first  $\frac{n}{2}$  coordinates yields an affine isomorphism 
of $\DP_n$ onto the convex hull of the rows of the $2n\times \frac{n^2}{2}$-matrix
\begin{align*}
  \left[
    \begin{array}{lllll}
      I&I&I&\ldots&I\\
      R^0&R^2&R^4&\ldots&R^{n-2}
    \end{array}\right]\,.\label{eq:qnpvert}
\end{align*}
In the representation given by this matrix let us partition the set of
$2n$ vertices (labelled from $0$ to $2n-1$) into two sets: consisting
of the $n$ rows with even index and the $n$ rows with odd index.
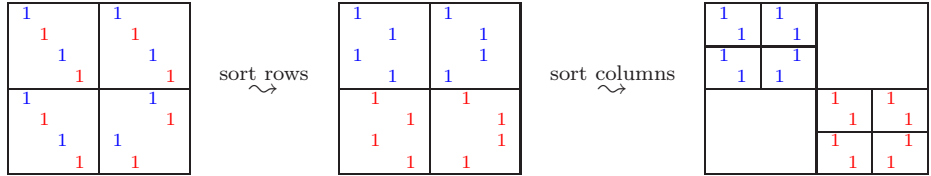
\newcommand{\even}{\color{red}}
\newcommand{\odd}{\color{blue}}
\begin{equation*} \tiny
  \begin{array}{|c@{\ }c@{\ }c@{\ }c|c@{\ }c@{\ }c@{\ }c|}
    \hline
    \odd 1&&& & \odd 1&&& \\
    &\even 1&& & &\even 1&& \\
    &&\odd 1& & &&\odd 1& \\
    &&&\even 1 & &&&\even 1 \\
    \hline
    \odd 1&&& & &&\odd 1& \\
    &\even 1&& & &&&\even 1 \\
    &&\odd 1& & \odd 1&&& \\
    &&&\even 1 & &\even 1&& \\
    \hline
  \end{array}
  \quad \overset{\text{sort rows}}{\leadsto} \quad
  \begin{array}{|c@{\ }c@{\ }c@{\ }c|c@{\ }c@{\ }c@{\ }c|}
    \hline
    \odd 1&&& & \odd 1&&& \\
    &&\odd 1& & &&\odd 1& \\
    \odd 1&&& & &&\odd 1& \\
    &&\odd 1& & \odd 1&&& \\
    \hline
    &\even 1&& & &\even 1&& \\
    &&&\even 1 & &&&\even 1 \\
    &\even 1&& & &&&\even 1 \\
    &&&\even 1 & &\even 1&& \\
    \hline
  \end{array}
  \quad \overset{\text{sort columns}}{\leadsto} \quad
  \begin{array}{|c@{\ }c|c@{\ }c|c@{\ }c|c@{\ }c|}
    \hline
    \odd 1&&\odd 1& & \multicolumn{4}{c|}{\ } \\
    &\odd 1&&\odd 1 & \multicolumn{4}{c|}{\ } \\
    \cline{1-4}
    \odd 1&&&\odd 1 & \multicolumn{4}{c|}{\ } \\
    &\odd 1&\odd 1& & \multicolumn{4}{c|}{\ } \\
    \hline
    \multicolumn{4}{|c|}{\ } & \even 1&&\even 1& \\
    \multicolumn{4}{|c|}{\ } & &\even 1&&\even 1 \\
    \cline{5-8}
    \multicolumn{4}{|c|}{\ } & \even 1&&&\even 1 \\
    \multicolumn{4}{|c|}{\ } & &\even 1&\even 1& \\
    \hline
  \end{array}
\end{equation*}
Then we permute the 
$\frac{n^2}{2}$ coordinates in such a way that in the first set of
rows (corresponding to even vertices) all nonzero entries are in the
first half (i.e., in the first $\frac{n^2}{4}$ columns). Then all
nonzero entries in the second set of rows (corresponding to the odd
vertices) will be in the second half (i.e., in the last
$\frac{n^2}{4}$ columns).  By a permutation of the coordinates within
the first half we get that the rows of even vertices yields precisely
the vertex set of $Q_{n/2} \times \{0\}$ (for $0 \in
\R^{\frac{n^2}{4}}$).  In the same way, the coordinates in the second
half can be permuted so that the rows of odd vertices equal the
vertices of $\{0\} \times Q_{n/2}$ (for $0 \in \R^{\frac{n^2}{4}}$). 
Since $0$ is not in the affine hull of $Q_{n/2}$, we deduce that
$\DP_n$ is a join of two copies of $Q_{n/2}$. Hence, its dimension
equals $2 \dim(Q_{n/2}) + 1 = 2 (n-2)+1=2n-3$ by Theorem~\ref{theo1}.
\end{proof}

\section{Lattice properties}

$\DP_n$ and $Q_n$ are {\em lattice polytopes}, i.e.\ their vertices
lie in the lattice $\Z^{n \times n}$ of integral vectors. It is readily
checked that all above affine isomorphisms respect lattice points. In
this section, we will show that these lattice polytopes have
especially nice properties which allow to completely describe their
Ehrhart $h^*$-vectors.

A $d$-dimensional lattice polytope $P$ containing $0$ in its interior
is {\em reflexive}, if its polar (or dual) polytope
\begin{align*}
  P^*:=\{x\in\R^d\mid \langle x,v\rangle\ge -1\; \forall\, v\in P\}
\end{align*}
is again a lattice polytope (in the dual lattice). This notion was
introduced by Batyrev in \cite{Batyrev94}. A generalization of this is
the class of Gorenstein polytopes.  A lattice polytope is a {\em
  Gorenstein polytope of codegree $k$}, if there is a positive integer
$k$ and an interior lattice point $m$ in $kP$ such that $kP-m$ is a
reflexive polytope. Such polytopes play an important role in the
classification of Calabi-Yau manifolds for string theory.  See
\cite{1161.14037} for basic properties. The next proposition tells us
that the polytopes $Q_n$ belong to this class. The {\em normalized
volume of $\R^n$}   is the volume form which assigns to the standard
simplex the volume $1$.

\begin{proposition}
  Let $n$ be arbitrary.  The polytope $Q_n$ is Gorenstein of codegree
  $n$ and normalized volume $n$.
\label{gorst}
\end{proposition}
\begin{proof}
  By Theorem~\ref{theo1}, the point $\frac1n(1,1,\ldots,1)$ is an
  interior point of $Q_n$ with equal integral distance $1/n$ to all
  facets, and $m:=(1,1,\ldots, 1)$ is the unique interior lattice
  point in $nQ_n$.  Hence $nQ_n-m$ is a reflexive polytope.

  By Theorem~\ref{theo1}, all facets of $Q_n$ are simplices of facet
  width $1$, hence they are all unimodular.  As we have seen,
  multiplying with n gives (up to translation) a reflexive polytope
  with the unique interior lattice point $m=(1,1,\ldots,1)$. The
  normalized volume of $nQ_n$ is the sum of the volumes of $n^2$
  pyramids over facets with apex $m$. But in $nQ_n$ each facet has
  normalized volume $n^{2n-3}$, and the apex has lattice distance one
  from the facet, so each pyramid has normalized volume
  $n^{2n-3}$. There are $n^2$ of these pyramids, so the normalized
  volume of $nQ_n$ equals $n^{2n-1}$. Dividing by
  $n^{2n-2}$ to get from $nQ_n$ back to $Q_n$ gives the normalized
  volume $n$ of $Q_n$.
\end{proof}
A polytope $P$ is \emph{compressed} if every so-called pulling
triangulation is regular and unimodular. Equivalently, $P$ is
compressed if for any supporting inequality $a^tx\le b$ with a
primitive integral normal $a$, i.e. with a normal vector whose entries
are integers and which is not an integral multiple of some other
integer vector, the polytope is contained in the set $\{x\mid b-1\le
a^tx\le b\}$.  For a more detailed explanation of these terms we refer
to~\cite{pre05768656}. This property has strong implications on the
associated toric ideal, see e.g.\ \cite{Sturmfels96}.  The next
proposition follows immediately from Theorem 1.1 of \cite{0984.13014}
and Theorem~\ref{theo1}.
\begin{proposition}
  Let $n$ be odd or even.  The polytope $Q_n$ is compressed.\qed
\end{proposition}
The \emph{Ehrhart polynomial} $L_P(k):=|kP\cap\Z^d|$ of a
$d$-dimensional lattice polytope counts the number of integral points in
integral dilates of $P$. It is well known that the generating function
of $L_P$ is given by
\begin{align*}
  \sum_{m\ge 0}L_P(m)t^m\ =\ \frac{h^*(t)}{(1-t)^{d+1}}
\end{align*}
for some polynomial $h^*$ of degree at most $d$ with integral
non-negative coefficients, see~\cite{BeckRobbins}.  Hence, determining
the Ehrhart polynomial is equivalent to finding the $h^*$-vector (also
called $\delta$-vector) of coefficients of $h^*(t)$. As is well-known,
$P$ is Gorenstein if and only if the $h^*$-vector is symmetric. The
following theorem shows that in our case this vector has a
particularly nice form.
\begin{theorem}
  Let $n$ be odd or even.  The $h^*$-vector of $Q_n$ satisfies $h^*_i=1$
  for $0\le i\le n-1$ and $h_i^*=0$ otherwise.
\label{theo3}
\end{theorem}
\begin{proof}
  Since the codegree of $Q_n$ is $n$ and its dimension is $2n-2$ by
  Theorem~\ref{theo1}, the maximal non-zero entry of the $h^*$-vector
  has to be $h^*_{n-1}$, see~\cite{BeckRobbins}.  By a theorem of
  Bruns and R{\"o}mer~\cite{1108.52013} we know that the $h^*$-vector
  of a Gorenstein polytope that has a regular unimodular triangulation
  is symmetric and unimodal. In particular, $h^*_i \ge 1$ for $i=0,
  \ldots, n-1$.  Since by Proposition~\ref{gorst} the sum of the
  entries of the $h^*$-vector equals $n$, the statement follows.
\end{proof}
In particular, if $n$ is odd, the previous result describes the
$h^*$-vector of $\DP_n$. Finally, let us deal with the even case.

\begin{corollary}
  Let $n$ be even.  The $h^*$-vector of $\DP_n$ equals
  \[(1, 2, 3, \ldots, \frac{n}{2}-1, \frac{n}{2}, \frac{n}{2}-1,
  \ldots, 2, 1).\] In particular, the polytope $\DP_n$ is Gorenstein
  of codegree $n$ and normalized volume $n^2/4$.\qed
\label{coro}
\end{corollary}

\begin{proof}

By the proof of Theorem~\ref{theo2}, $\DP_n$ is given up to coordinate permutation as 
the convex hull of the rows of\begin{align*}
  \left[
    \begin{array}{ll}
      \tilde{W}&0\\
      0&\tilde{W}
    \end{array}\right]\,,
\end{align*}
where $\tilde{W}$ is the $n \times (\frac{n}{2})^2$-matrix whose rows
are the vertices of $Q_{\frac n2}$ as given in (\ref{eq:qnvert}). The
integral linear functional which sums the first $\frac n2$ coordinates
evaluates to $1$ on the first $\frac n2$ rows, and to $0$ on the
second half. Hence, the two copies of $Q_{\frac n2}$ (say, $P_1 \times
\{0\}$ and $\{0\} \times P_2$) have lattice distance $1$ in the
lattice $\Z^{\frac{n^2}{2}}\cap\aff \DP_n$. In other words, there is
an affine isomorphism respecting lattice points which maps $\DP_n$
onto the convex hull of $P_1 \times \{0\} \times \{1\}$ and $\{0\}
\times P_2 \times \{0\}$ in $\R^{\frac{n^2}{2}+1}$. Therefore, the
statement follows from the well-known fact
\cite[Ex. 3.32]{BeckRobbins}  
that in this case the $h^*$-polynomial equals the product of the
$h^*$-polynomials of $P_1$ and $P_2$.
\end{proof}
\section{Substructures}

In ~\cite{Baumeister2007} the authors discussed which subgroups of a
permutation group yield faces of $P(G)$. An obvious class of such subgroups are
stabilizers: 

Take a partition $[n] := \{1, \ldots, n\} = \bigsqcup I_i$. Then the polytope of the
stabilizer of the subsets $I_i$
$$\operatorname{stab}(G;(I_i)_i) := \{
\sigma \in G : \sigma(I_i)=I_i \text{ for all } i \} \leq G$$ is a face of
$P(G)$. The authors conjecture that there are no other examples.
\medskip\\
{\bf Conjecture 5.8} ~\cite{Baumeister2007}
    Let $G \le S_n$. Suppose $H \le G$ is a subgroup such that $P(H)
    \preceq P(G)$ is a face. Then $H = \operatorname{stab}(G;(I_i)_i)$
    for a partition $[n] = \bigsqcup I_i$.
\medskip\\

 We have verified the conjecture for $G = S_n$ as well as for $G\leq S_n$ a cyclic group, see Proposition~5.9
\cite{Baumeister2007}. Meanwhile Jessica Nowack and Daniel Heinrich studied this question for the dihedral
groups in their Diploma theses.

\begin{prop}\cite{No11, Hei11}
 The Conjecture 5.8 holds for $G = D_n \leq S_n$ for every  $n$.
\end{prop}
{\em Sketch of the proof.}
  For $n$ odd Heinrich first shows that, if $H$ is the subgroup of all
  rotations of $G$, then $P_H$ is not a face of $P_G$.
  The remaining subgroups are precisely the stabilizers of their
  orbits, see  Theorem~7.1.1 of \cite{Hei11}.

  For $n$ even the main work is to show that the subgroup of all
  rotations, the subgroup of the squares of the rotations and finally 
  the subgroup generated by the squares of the rotations and by the
  reflections through two edges are precisely those subgroups $H$ of
  $G$ for which $P_H$ is not a face of $P_G$. Nowack shows that the 
  remaining subgroups are precisely the stabilizers of their
  orbits, see Section~4.2 of \cite{No11}. \hfill \qed

\bibliographystyle{alpha}
\bibliography{dihedral2}

\end{document}